\documentclass[letter,12pt]{amsart} 
\usepackage[svgnames]{xcolor} 

\usepackage[utf8]{inputenc} 
\usepackage[OT2,T1]{fontenc}
\DeclareSymbolFont{cyrletters}{OT2}{wncyr}{m}{n}
\DeclareMathSymbol{\Sha}{\mathalpha}{cyrletters}{"58}

\usepackage{PTSerif} 
\usepackage[all,cmtip]{xy}\xyoption{dvips}

\usepackage[linktocpage]{hyperref} 
\usepackage{amsmath}
\usepackage{amssymb}
\usepackage{amsthm}
\usepackage{blindtext}
\usepackage[margin = 1 in]{geometry}
\usepackage{mathtools}
\usepackage{graphicx}
\usepackage{tikz-cd}
\usepackage{footnote}
\usepackage{fancyhdr}
\usepackage{aliascnt}
\usepackage{hyperref}
\usepackage{stmaryrd}
\usepackage{comment}
\usepackage{mathrsfs}
\usepackage{cleveref}
\usepackage{float}
\usepackage{todonotes}

\usepackage[
backend=biber,
style=alphabetic,
]{biblatex}
\makeatletter
\providecommand\@dotsep{5}

\makeatother

\usepackage[normalem]{ulem}
\hypersetup{
    colorlinks=true,
    linkcolor={blue!50!black},
    citecolor={green!50!black},
    urlcolor={red!80!black}
}

\newcommand{\kdot}{{{\,\begin{picture}(1,1)(-1,-2)\circle*{2}\end{picture}\,}}}

\theoremstyle{plain}
\newtheorem{theorem}{Theorem}[section]
\newaliascnt{cor}{theorem}
\newaliascnt{lem}{theorem}
\newaliascnt{prop}{theorem}
\newaliascnt{assume}{theorem}
\newaliascnt{notat}{theorem}
\newaliascnt{rem}{theorem}
\newaliascnt{quest}{theorem}
\newtheorem{corollary}[cor]{Corollary}
\newtheorem{lemma}[lem]{Lemma}
\newtheorem{proposition}[prop]{Proposition}

\theoremstyle{definition}
\newtheorem{definition}[theorem]{Definition}

\newtheorem{example}[theorem]{Example}

\newtheorem{notation}[notat]{Notation}

\theoremstyle{remark}
\newtheorem{remark}[rem]{Remark}

\newtheorem{qu}[quest]{Question}

\usepackage[all,cmtip]{xy}\xyoption{dvips}

\graphicspath{ {images/} }

\newcommand{\mb}{\mathbb}

\newcommand{\bC}{\mb C}

\newcommand{\N}{\mb{N}}

\renewcommand{\O}{\mathscr{O}}

\newcommand*{\sheafhom}{\mathcal{H}\kern -.5pt om}

\newcommand{\du}{\underline{\Omega}}

\DeclareMathOperator{\id}{id}

\DeclareMathOperator{\Cone}{Cone}

\renewcommand{\to}{\rightarrow}

\DeclareMathAlphabet{\mathchanc}{OT1}{pzc}%
                                 {m}{it}

\newcommand{\mcL}{\mathchanc{L}}

\newcommand{\mcR}{\mathchanc{R}}

\counterwithin{equation}{theorem}

\title{General base change for relative Du Bois complexes}
\author{Caleb Ji and S\'andor J Kov\'acs}
\date{\today}
\thanks{Caleb Ji was supported in part by NSF Grant DGE-2036197} 
\address{CJ: Columbia University, Department of Mathematics, New York, NY 10027, USA}
\email{caleb.ji@columbia.edu} 
\thanks{S\'andor Kov\'acs was supported in part by NSF Grant DMS-2100389.} 
\address{SJK: University of Washington, Department of Mathematics, 
Seattle, WA 98195, USA}
\email{skovacs@uw.edu}

\begin{document}
\begin{abstract}
    A partial answer is given to a question raised by Kov\'acs and Taji in \cite{sk23}, namely that the relative Du~Bois complex of a family parametrized by a non-singular curve commutes with base change to a general point on the base. It is also shown that this property usually fails for special points. 
\end{abstract}
\maketitle 

\section{Introduction}\noindent
The sheaf of K\"ahler differentials plays an important role in the understanding of the geometry of complex manifolds. Along with its exterior powers it forms the de~Rham complex, a filtered complex that leads to the Hodge-to-de~Rham spectral sequence and to many applications. 

Unfortunately, for singular varieties, these sheaves and the de~Rham complex do not have the same useful properties. However, there is a replacement, the \emph{Deligne-Du~Bois complex} (usually simply referred to as the \emph{Du~Bois complex}) \cite{DB}, which has many of the useful properties of the de~Rham complex in the non-singular case.
For instance it also leads to a spectral sequence similar to the Hodge-to-de~Rham one and it also degenerates at the $E_1$ stage, so it also leads to Hodge theoretic applications. 

The relative situation is more complicated and troublesome. It seems reasonable that one would like to have an object in the relative case similar to the Du~Bois complex. In fact, such an object has been constructed in the case when the base is a non-singular curve \cite{sk1,sk23}.
The sheaf of relative K\"ahler differentials is invariant under base change. This property is a very useful tool in many situations, for instance in the theory of variations of Hodge structure. It would be extremely useful if the same were true for the \emph{relative Du~Bois complex}. We do not expect this to hold in general, but that only makes it more important to determine a reasonable set of assumptions that would imply this. In other words, we are interested in the following:
\begin{qu}[\protect{\cite[Problem~5.2]{sk23}}]\label{quest:restrict}
    Let $f\colon X\rightarrow C$ be a morphism from a complex variety $X$ to a smooth complex curve $C$ of relative dimension $n$.  Let $\du^p_{X/C}$ be the $p^\text{th}$ relative Du~Bois complex, constructed in  \cite{sk1}. Under what conditions does base change hold for $\du^p_{X/C}$? More precisely, let $c$ be a closed point of $C$ giving rise via base change to the closed embedding $\jmath_c\colon X_c\hookrightarrow X$. Under what conditions does the isomorphism 
    \[
    \mcL\jmath_c^*\du_{X/C}^p\simeq \du_{X_c}^p
    \]
    hold? 
\end{qu}

\begin{remark}\label{rem:restriction-to-open}
    It is straightforward and left to the reader that the construction of  $\du^p_{X/C}$ commutes with open embeddings both on the base and on the source, cf.~\cite[1.3.6]{sk1}.
\end{remark}

\noindent
In this paper we prove the following partial answer to this question using the notation introduced in \autoref{quest:restrict}:
\begin{theorem}[=\protect{\autoref{corollary: generic bc}}]\label{main-thm}
    Let $f\colon X\rightarrow C$ be a morphism from a complex variety $X$ to a smooth complex curve $C$.  Then there exists a non-empty open subset $U\subseteq C$ such that for each $c\in U$, 
    \[
    \mcL\jmath_c^*\du^p_{X/C} \simeq \du^p_{X_c}.
    \]
\end{theorem}

\subsection*{Acknowledgement}
We would like to thank the referee for useful comments and corrections, and in particular for suggesting a simplification in the proof of Example~\ref{ex}

\section{Cubical hyperresolutions}\noindent
An important ingredient of Hodge theory of singular varieties is the notion of hyperresolutions. 
This is a simplicial or cubical variety consisting of smooth complex varieties such that the natural map from its geometric realization to the original singular variety is proper and has contractible fibers. In particular, this natural map is of cohomological descent cf.~\cite[Defs.~5.1,5.6,5.10,Prop.~5.13]{peters}.
A particular flavor of hyperresolutions was developed in \cite{Carlson85} and \cite{GNPP88} independently. Here we follow the terminology of \cite{GNPP88}, where it is called a \emph{cubical hyperresolution} (it is essentially the same as what \cite{Carlson85} calls a \emph{smooth polyhedral resolution}). In the sequel, when we say hyperresolution, we mean a cubical hyperresolution. In particular, we will be using hyperresolutions consisting of finitely many objects. We will also use the following weaker notion:
\begin{definition}[\protect{\rm{\cite[Def.~3.2]{Kovacs25a}}}]\label{def:cubical-partial-hyperresolution} 
    Let $X$ be a scheme of finite type over $\bC$.  A \emph{(cubical) partial hyperresolution} of $X$ is a cubical variety $\varepsilon_\kdot\colon X_\kdot\to X$  that has all the properties of a \emph{cubical hyperresolution}  \cite[Def.~5.10]{peters}, except that the individual varieties $X_\alpha$ are not assumed to be nonsingular.  More precisely, it is a $\square^+_r$ scheme over $X$ for some $r\in\N$ as in \cite[Def.~I.2.12]{GNPP88}, cf.~\cite[Def.~2.12]{kov-schw}, of \emph{cohomological descent} \cite[{\S}5.3]{HTIII},\cite[Defs.~5.6,5.10]{peters}.  This is called a \emph{polyhedral resolution} (as opposed to a \emph{smooth polyhedral resolution}) in \cite[p.~596]{Carlson85}. 
\end{definition}

Recall that the standard way to define the Du~Bois complex of a complex scheme $X$ is as follows. Let $\varepsilon_\kdot\colon X_\kdot\to X$ be a hyperresolution. Then
\begin{equation}
    \label{eq:38}
    \du^p_X\colon=\mcR{\varepsilon_\kdot}_*\Omega^p_{X_\kdot}.
\end{equation}

\section{The relative Du~Bois complex}\noindent
The following notation will be used throughout:
\begin{notation}
    Let $f\colon X\rightarrow C$ be a morphism from a complex variety $X$ to a smooth complex curve $C$ of relative dimension $n$. Further let $c\in C$ be a closed point and $X_c=f^{-1}(c)$ the fiber of $f$ over $c$. We will use $\jmath_c$ to denote the closed embedding $X_c\hookrightarrow X$. Finally,  $\du^p_{X/C}$ will denote the $p^\text{th}$ relative Du~Bois complex constructed in  \cite{sk1}. 
\end{notation}

First, let us record a simple observation on the naturality of the construction of $\du^p_{X/C}$. It was already proved in \cite[p.~6]{sk1} that the construction is independent of the hyperresolution used and it has an interesting consequence that will help us later.

\begin{lemma}\label{lem:partial-hyperres}
    Let $\varepsilon_\kdot\colon X_\kdot\to X$ be a partial hyperresolution as in Definition~\ref{def:cubical-partial-hyperresolution}. Then
    \begin{equation*}
        \du^p_{X/C}\simeq \mcR{\varepsilon_\kdot}_*\du^p_{X_\kdot/C}
    \end{equation*}
\end{lemma}
\begin{proof}
    The proof is very similar to that of \cite[V.3.6(5)]{GNPP88}.

    Choose a cubical hyperresolution of $X_\alpha$ for each $\alpha$ in the index set of $X_\kdot$. Then these may be combined to form a cubical hyperresolution of $X$.  Denote the resulting cubical hyperresolution by $\varepsilon'_\kdot\colon X'_\kdot\to X$. (Note that the usual proof of the existence of hyperresolutions produces a hyperresolution like this due to its construction). Then one may replace $\du_{X_\alpha}^\kdot$ with its defining expression as in \eqref{eq:38} using this particular hyperresolution. The composition of these (following the rules of cubical schemes) gives the analogous expression for $X$. In other words, (as proved in \cite[V.3.6(5)]{GNPP88}),
    \begin{equation}
    \label{eq:partial}
        \du^p_{X}\simeq \mcR{\varepsilon_\kdot}_*\du^p_{X_\kdot}.
    \end{equation}
    We will use the construction of $\du^p_{X/C}$ in \cite{sk1} and the notation used there. 
Recall that $K_p^{\kdot}$ denotes an explicit incarnation of $\du^p_X = \mcR{\pi_\kdot}_*\Omega^p_{X_{\kdot}}$ in the derived category given by complexes of sheaves of $C^\infty$ differential forms using a given hyperresolution. It is proved in \cite{sk1} that the construction is independent of the hyperresolution chosen. The next step in the construction is defining $M_p^{\kdot}$, which is used to define $\du^p_{X/C}$.
$M_p^\kdot$ is defined recursively, starting with $p=n-1$ and then defining $M_p^\kdot$ using the definition of $M_{p+1}^\kdot$ and certain morphisms connecting the $M$'s and $K$'s.  In particular, there are morphisms, also defined recursively as outlined above,  
 \[
 w_{p+1}''\colon K_{p+1}^{\kdot}\otimes f^*\omega_C \rightarrow M_{p+1}^{\kdot}\otimes f^*\omega_C
 \]
 and $M_{p}^{\kdot}$ 
 is defined as 
 \begin{equation}
     \label{def-of-M_p}
    M_p^\kdot\colon=\Cone(w_{p+1})[-1]\otimes (f^*\omega_C)^{-1},
 \end{equation}
 where $w_{p+1} = w_{p+1}''\otimes \id_{f^*\omega_C^{-1}}\colon K_{p+1}^\kdot\to M_{p+1}^\kdot$.  We refer to \cite{sk1} for more details. 
    
    Here we will use the cubical hyperresolution constructed above and used in \eqref{eq:partial}.
    Observe that the $K_p^{\kdot}$'s constructed on $X$ and the ones constructed on the $X_\alpha$'s satisfy the same relationship as the classes they represent do in \eqref{eq:partial}. Then, using \autoref{def-of-M_p} and descending induction, it follows that the same relationship as in \eqref{eq:partial} holds for the $M_p^\kdot$'s constructed on $X$ and the $X_\alpha$ respectively. Finally, because $\du^p_{X/C}$ is defined as the equivalence class of $M_p^\kdot$ in the derived category, this implies the desired statement.
\end{proof}

\section{Base change} 
\begin{definition} 
Let $g\colon Y\rightarrow B$ be a morphism of complex schemes $Y$ and $B$. We will say that $g$ admits a \emph{simultaneous relative hyperresolution} if there exists a hyperresolution ${\pi_\kdot}\colon Y_{\kdot}\rightarrow Y$ such that for each $\alpha$ in the index set of $Y_\kdot$, 
$q_\alpha=g\circ\pi_\alpha$ (as below) is smooth. 
\[\begin{tikzcd}
	{Y_{\alpha}} \\
	Y & B
	\arrow["\pi_\alpha"', from=1-1, to=2-1]
	\arrow["q_\alpha", from=1-1, to=2-2]
	\arrow["g"', from=2-1, to=2-2]
\end{tikzcd}\]
\end{definition}
\noindent
If a morphism admits a simultaneous relative hyperresolution, then the same formula as in \eqref{eq:38} applies for $\du^p_{X/C}$ by \autoref{lem:partial-hyperres}:
\begin{proposition}
\label{prop: quop}
Let $f\colon X\rightarrow C$ be a morphism from a complex variety $X$ to a smooth complex curve $C$ and assume that $f$ admits a simultaneous relative hyperresolution.  Then \[\du^p_{X/C}\simeq \mcR{\pi_\kdot}_*\Omega^p_{X_{\kdot}/C}.\]
\end{proposition}
\begin{proof}
 By \autoref{lem:partial-hyperres}, $\du^p_{X/C}\simeq \mcR{\pi_\kdot}_*\du^p_{X_{\kdot}/C}$ and because $\pi_\kdot\colon X_\kdot\to X$ is a  simultaneous relative hyperresolution, $\du^p_{X_{\alpha}/C}\simeq\Omega^p_{X_{\alpha}/C}$ for each $\alpha$ in the index set of $X_\kdot$.
\end{proof} 

\noindent
We will also need the following simple fact:
\begin{lemma}
    \label{lemma:tor-independence}
    Let $X$ be a scheme, $H\subseteq X$ an effective Cartier divisor, and $\pi\colon\widetilde X\to X$ a morphism from a reduced irreducible scheme $\widetilde X$ such that $\pi(\widetilde X)\not\subseteq H$. Then $H$ and $\widetilde X$ are Tor-independent with respect to $X$.  
\end{lemma}
\begin{proof}
    We need to prove that 
    \[
    \mcL\pi^*\O_H\simeq\pi^*\O_H.
    \]
    The standard short exact sequence,
    \[
    \xymatrix{
    0\ar[r]&\O_X(-H)\ar[r]& \O_X\ar[r] & \O_H \ar[r] & 0,
    }
    \]
    gives a two term resolution of $\O_H$ and shows that $\mcL\pi^*\O_H\simeq\big[\pi^*\O_X(-H)\to\pi^*\O_X\big]$.

    Next, let $U\colon=X\setminus H$ and observe that the morphism $\pi^*\O_X(-H)\to\pi^*\O_X$ is an isomorphism on $\pi^{-1}U$, which is a dense open subset of $\widetilde X$ by assumption. This implies that the kernel of $\pi^*\O_X(-H)\to\pi^*\O_X$ is a torsion $\O_{\widetilde X}$-module.
    
    On the other hand, $\pi^*\O_X(-H)$ is locally isomorphic to $\O_{\widetilde X}$, in particular it is torsion-free. It follows that the morphism $\pi^*\O_X(-H)\to\pi^*\O_X$ is injective and hence, as a complex, quasi-isomorphic to its cokernel, which is isomorphic to $\pi^*\O_H$. This proves the desired statement.
\end{proof}

\noindent
Now we are ready to prove the key base change result. 

\begin{theorem}
\label{prop: pb}
Let $f\colon X\rightarrow C$ be a morphism from a complex variety $X$ to a smooth complex curve $C$ and assume that $f$ admits a simultaneous relative hyperresolution.
Then for each closed point $c$ of $C$ with fiber $X_c$, we have 
\[
\mcL\jmath_c^*\du_{X/C}^p \simeq \du^p_{X_c}.
\]
\end{theorem} 
\begin{proof}
Consider the diagram given by the presumed simultaneous relative hyperresolution, ${\pi_\kdot}\colon X_{\alpha}\rightarrow X$. Note that we may assume that $X_\alpha$ is irreducible for each $\alpha$.
\[\begin{tikzcd}
	{(X_{\alpha})_c} & {X_{\alpha}} \\
	{X_c} & X \\
	c & C
	\arrow["\jmath_{c,\alpha}", from=1-1, to=1-2]
	\arrow["{{\pi_\kdot}_c}"', from=1-1, to=2-1]
	\arrow["{\pi_\kdot}", from=1-2, to=2-2]
	\arrow["\jmath_c", from=2-1, to=2-2]
	\arrow["{f_c}"', from=2-1, to=3-1]
	\arrow["f", from=2-2, to=3-2]
	\arrow[from=3-1, to=3-2]
\end{tikzcd}\]
It follows by the assumption that $\pi_\kdot$ is a simultaneous relative hyperresolution, that the base change ${\pi_c}_\kdot\colon (X_{\kdot})_c\rightarrow X_c$ is also a hyperresolution.  Then $\du^p_{X_c}\simeq \mcR({\pi_\kdot}_c)_*\Omega^p_{(X_{\kdot})_c}$ by \eqref{eq:38}.  
On the other hand,  $\du^p_{X/C} \simeq \mcR{\pi_\kdot}_*\Omega^p_{X_{\kdot}/C}$ by \autoref{prop: quop}.  
Observe that the fact that $X_\alpha$ is smooth over $C$ implies that $\pi_\alpha(X_\alpha)\not\subseteq X_c$ for any $c\in C$ closed point. Because $C$ is a smooth curve, each closed point $c\in C$ is an effective Cartier divisor on $C$ and hence $X_c$ is an effective Cartier divisor on $X$. Then it follows by \autoref{lemma:tor-independence} that $X_{\alpha}$ and $X_c$ are Tor-independent over $X$.  Therefore we may apply the base change statement \cite[\href{https://stacks.math.columbia.edu/tag/08IB}{08IB}]{stacks-project} to the top square, and hence we conclude that
\[
\mcL\jmath_c^*\du_{X/C}^p \simeq \mcL\jmath_c^*\mcR{\pi_\kdot}_*\Omega^p_{X_{\kdot}/C} \simeq \mcR{{\pi_c}_\kdot}_*\mcL\jmath_{c,\kdot}^*\Omega^p_{X_{\kdot}/C}.
\]
Furthermore, $X_\alpha$ being smooth over $C$ implies that $\Omega^p_{X_{\alpha}/C}$ is locally free on $X_{\alpha}$. 
Thus, $\mcL\jmath_{\alpha}^*\Omega^p_{X_{\alpha}/C} = \jmath_{\alpha}^*\Omega^p_{X_{\alpha}/C} \simeq \Omega^p_{(X_{\alpha})_c}$, and so
\[
\mcL\jmath_c^*\du_{X/C}^p \simeq \mcR{{\pi_c}_\kdot}_*\Omega^p_{(X_{\alpha})_c} \simeq \du^p_{X_c},
\]
as desired. 
\end{proof}

\noindent
\autoref{prop: pb} easily implies our main result, \autoref{main-thm}:
\begin{corollary}[=\protect{\autoref{main-thm}}]
\label{corollary: generic bc}
Let $f\colon X\rightarrow C$ be a morphism from a complex variety $X$ to a smooth complex curve $C$.  Then there exists a non-empty open subset $U\subseteq C$ such that for each $c\in U$, 
\[
    \mcL\jmath_c^*\du^p_{X/C} \simeq \du^p_{X_c}.
\]
\end{corollary}
\begin{proof}
    Note that the statement is vacuously true if $f$ is not dominant. 
 Otherwise, let $\pi_\kdot\colon X_{\kdot}\to X$ be a finite cubical hyperresolution. 
     Because each $X_\alpha$ is smooth, by generic smoothness there exists a non-empty open $U_\alpha\subseteq C$ such that $f\circ{\pi_\alpha}\colon (f\circ{\pi_\alpha})^{-1}(U_\alpha)\rightarrow U_\alpha$ is smooth. 
     Take $U$ to be the intersection of the $U_\alpha$, which is open as the hyperresolution is finite.
     Then  \autoref{prop: pb} implies that base change holds over $U$, as desired, cf.~\autoref{rem:restriction-to-open}. 
\end{proof}

Kov\'acs and Taji constructed the full relative Du Bois complex $\du_{X/C}^\kdot$ over a curve whose graded pieces are given by the $\du^p_{X/C}$ above in \cite{sk23}.  They posed the problem of finding suitable conditions under which the formation of $\du_{X/C}^\kdot$ commutes with the base change to a closed point, cf.~\autoref{quest:restrict}, \cite[Problem 5.2]{sk23}.  
The proof of \autoref{prop: quop} essentially verbatim proves the same statement for $\du_{X/C}^\kdot$ and  combined with the construction of $\du_{X/C}^\kdot$ in \cite{sk23} and the proofs of \autoref{prop: pb} and \autoref{corollary: generic bc}, it leads easily to the following corollary:
\begin{corollary}
Let $f\colon X\rightarrow C$ be a morphism from a complex variety $X$ to a smooth complex curve $C$ and assume that $f$ admits a simultaneous relative hyperresolution. Then 
\[
\mcL\jmath_c^*\du^\kdot_{X/C}\simeq \du^\kdot_{X_c}.
\]
    In particular, for an arbitrary $f$, there exists a non-empty open subset $U\subseteq C$ such that for each $c\in U$, 
\[
    \mcL\jmath_c^*\du^\kdot_{X/C} \simeq \du^\kdot_{X_c}.
\]
\end{corollary}

We will now give an example of interest in which base change does not hold.

\begin{example}\label{ex}
Let $f\colon X\rightarrow C$ be a morphism of relative dimension $n$ from a smooth complex variety $X$ to a smooth complex curve $C$. Assume that $f\colon X\rightarrow C$  is smooth outside a single point $0\in C$, over which the fiber $\imath\colon Y=X_0\hookrightarrow X$ is a divisor with simple normal crossings.
In this scenario we claim that base change does not hold, and in particular, 
\[
\mcL\imath^*\du_{X/C}^{n}\not\simeq \du_{Y}^{n}.
\]

To show this, by \cite[Example 7.23]{peters}, we may choose the single sheaf $\Omega^p_X/I_Y\Omega^p_X(\log Y)$ as an explicit representative for $\imath_*\du^p_Y$.
First, pull-back the fundamental distinguished triangle  to $Y$:
\[
\xymatrix{\du_{X/C}^{p-1}\otimes f^*\omega_C\ar[r] &\Omega^p_X\ar[r] &\du_{X/C}^p\ar[r]^-{+1}&}
\]
This gives the distinguished triangle 
\[
\xymatrix{\mcL\imath^*\du_{X/C}^{p-1}\ar[r] & \imath^*\Omega^p_X \ar[r] & \mcL\imath^*\du_{X/C}^p\ar[r]^-{+1}&.} 
\]
Here we have used the fact that $\omega_C$ is locally free and that $f\circ\imath$ maps $Y$ to a point, so $\mcL\imath^*f^*\omega_C\simeq\O_Y$.
Now observe that if the construction of  $\du_{X/C}^p$ commuted with base change to this point, then one would have the following distinguished triangle:
\begin{equation}
    \label{false-dt}
    \xymatrix{\du_{Y}^{p-1}\ar[r] & \imath^*\Omega^p_X \ar[r] & \du_{Y}^p\ar[r]^-{+1}&.} 
\end{equation}
    Taking $p=n+1=\dim X$, and observing that then $\du_{Y}^{n+1}=0$, \eqref{false-dt} would imply that $\du_{Y}^{n}\simeq i^*\Omega^{n+1}_X$. In order to prove our claim, it suffices to show that $\imath^*\Omega^{n+1}_X$ and $\du_Y^{n}$ are not isomorphic.  

    Recall that there always exist cubical hyperresolutions, in which the first morphism is a resolution and all subsequent morphisms are from varieties of smaller dimension cf.~\cite[Thm.~5.26]{peters}. This implies that if $\nu:\widetilde Y\to Y$ is the normalization (i.e., a resolution) of $Y$, then $\du_Y^n\simeq\nu_*\omega_{\widetilde Y}$ cf.~\cite[p.~335]{steenbrink1}. As $\widetilde Y$ is smooth and $\nu$ is finite, $\nu_*\omega_{\widetilde Y}$ is locally isomorphic to $\nu_*\O_{\widetilde Y}$, which is \emph{not} isomorphic to $\O_Y$ near points where $Y$ is not normal. 

    On the other hand, as $X$ is smooth, $\imath^*\Omega^{n+1}_X$ is a line bundle and hence locally isomorphic  to $\O_Y$ everywhere. This implies that $\imath^*\Omega^{n+1}_X$ and $\du_Y^{n}$ are indeed not isomorphic, as desired.
\end{example}

\printbibliography 
\end{document}